\documentclass[
11pt,                          
draft,                         
english                        
]{article}

\usepackage[english]{babel}    
\usepackage{amsmath}           
\usepackage[utf8]{inputenc}    
\usepackage[T1]{fontenc}       
\usepackage{longtable}         
\usepackage{exscale}           
\usepackage[final]{graphicx}   
\usepackage[sort]{cite}        
\usepackage{array}             
\usepackage{wasysym}           
\usepackage[a4paper]{geometry} 
\usepackage[multiuser]{fixme}  
\usepackage{xspace}            
\usepackage{tikz}              
\usepackage{ifdraft}           
\usepackage[expansion=false    
           ]{microtype}        
\usepackage[nottoc]{tocbibind} 
\usepackage[backref=page,      
            final=true,         
           pdfpagelabels       
           ]{hyperref}         
\usepackage[all]{xy}
\usepackage{slashed}
\usepackage{faktor}
\usepackage{chairx}
\usetikzlibrary{matrix}
\usepackage{tikz-cd}
\usepackage[amsmath,thmmarks,hyperref]{ntheorem}
\usepackage{bbm}
\ifdraft{\synctex=1}{}

\fxusetheme{color}

\newcommand{\bla}{\langle \hspace{-2.7pt} \langle}
\newcommand{\bra}{\rangle\hspace{-2.7pt} \rangle}

\author{
\textbf{Jonas Schnitzer}\thanks{\texttt{jschnitzer@unisa.it}}\\[0.5cm]
 Dipartimento di Matematica \\
  Università degli Studi di Salerno \\
  Via Giovanni Paolo II, 132\\
  84084  Fisciano (SA) \\
  Italy}

\title{Regular Jacobi Structures and Generalized Contact Bundles}

\begin{document}

\maketitle

\begin{abstract}
A Jacobi structure $J$ on a line bundle $L\to M$ is weakly regular if the sharp map $J^\sharp : J^1 L \to DL$ has constant 
rank. A 
generalized contact bundle with regular Jacobi structure possess a transverse complex structure. Paralleling the work of 
Bailey in generalized complex geometry, we find condition on a pair consisting of a regular Jacobi structure and an 
transverse complex structure to come from a generalized contact structure. In this way we are able to construct interesting 
examples of generalized contact bundles. As applications: 1) we prove that every 5-dimensional nilmanifold is equipped with 
an invariant generalized contact structure, 2) we show that, unlike the generalized complex case, all contact bundles over a 
complex manifold possess a compatible generalized contact structure. Finally we provide a counterexample presenting a locally 
conformal symplectic bundle over a generalized contact manifold of complex type that do not possess a compatible generalized 
contact structure.
\end{abstract}

\addtocontents{toc}{\protect\setcounter{tocdepth}{1}}
\tableofcontents

\section{Introduction}
Generalized geometry was set up in the early 2000's by Hitchin \cite{H2003} and further developed 
by Gualtieri \cite{zbMATH05960697}, 
and the literature about them is  now rather wide.
Generalized complex manifolds encompass symplectic 
and complex manifolds, i.e. they provide a common framework for symplectic and complex manifolds. 
Generalized complex structure are even dimensional, so from the very beginning people searched for their odd 
dimensional counterpart. The odd dimensional counterparts of symplectic manifolds are contact manifolds and 
actually this gave a hint: the odd dimensional analogue of generalized complex structures should at least contain 
contact manifolds. First attempts to give meaningful definitions have been done by Iglesias-Ponte and Wade 
\cite{IGLESIASPONTE2005249}, Poon and Wade \cite{JLMS:JLMS0333} and  Aldi and Grandini \cite{ALDI201578}. 
Recently, Vitagliano and Wade proposed a definition \cite{GenConBun} that seems to be well 
motivated and has the main advantage that all the previous attempts are included. 
This is the definition we will work with. 
Generalized contact structures (in the sense of Vitagliano and Wade) 
encompass non necessarily co-orientable contact structures and integrable 
complex structures on the \emph{Atiyah-algebroid} of a line bundle. 
In \cite{2017arXiv171108310S} it was proven that, locally, generalized contact structures are products of either a contact and a homogeneous        
generalized complex or a symplectic and a generalized contact structure (of a simpler type). 
This local normal form theorem can be seen as a manual to build up 
examples which are not included in the extreme cases by taking products. On the other hand, it is interesting to find 
examples which do not belong to this class, i.e. are not globally products. 
This is the aim of this work: we want to describe a method, similar to the one described in \cite{GenComNil} for generalized 
complex structures, to detect whenever a \emph{weakly regular} Jacobi structure can 
come from a generalized contact bundle. We use then using our method to find examples, 
which do not come from a global product. 
 
This note is organized as follows. First we recall the arena for generalized contact bundles, the omni-Lie algebroid, 
and the very definition of generalized contact structures. 
This is far from being complete and is more meant to fix notation than 
to give a proper introduction to the topic. A more detailed discussion of these notions can be found in \cite{DirJacBun}.  
The second step is to introduce the notion of weakly regular Jacobi
structures, which are the equivalent notion of a regular Poisson structures in Poisson geometry, \emph{transversally complex 
subbundles}, which are basically complex structures on a normal bundle, and their connection to generalized contact 
structures. The general idea is that every generalized contact structure with weakly regular Jacobi structure comes together 
with a transversally complex subbundle, 
but the converse is not true in general. We will see that we can find necessary and sufficient condition to find the converse
statement, which will be a cohomological obstruction in some spectral sequence.

In the last part we construct examples and a counterexample, with the help of the previous part. 
This part is in turn divided in 3 parts.  First of all,
we prove that every five dimensional nilpotent Lie group possesses an
invariant generalized contact structure. This can be seen as the odd dimensional analogue of the existence of generalized complex structures on six dimensional 
nilpotent Lie groups proven in \cite{SymplFol}. 
Our second examples are \emph{contact fiber bundles} over a complex base. It turns out that they posses generalized contact structures, since the obstructions of the previous parts become trivial.

Finally, we discuss a specific counterexample, i.e. an example where the obstructions do not vanish.
  
\section{Preliminaries and Notation}
This introductory section is divided into two parts: first we recall the Atiyah algebroid of a vector bundle and the 
corresponding $Der$-complex with applications to contact and Jacobi geometry. Afterwards, we introduce the arena for 
generalized geometry in odd dimensions, the omni-Lie algebroids, and give a quick reminder of generalized contact bundles
together with the properties we will need afterwards.    

\subsection{The Atiyah algebroid and the $\mathrm{Der}$-complex}
The notions of  Atiyah algebroid of a vector bundle and the associated $Der$-complex are known and are used in many 
other situations. This section is basically meant to fix notation. A more complete introduction to this  can be found in 
\cite{GenConBun}, which also discusses the notion of Dirac structures on the \emph{omni-Lie algebroid} of line bundles in more detail, nevertheless 
the notion of Omni-Lie algebroids was first defined in \cite{CHEN2010799}, in order to study Lie algebroids and local Lie 
algebra structures on vector bundles.

\begin{definition}
Let $E\to M$ be a vector bundle. A derivation is a map $\Delta\colon \Secinfty(E)\to \Secinfty(E)$, fulfilling 
	\begin{align*}
	\Delta(fe)=X(f)e+f\Delta(e) \ \forall f\in\Cinfty(M)\ \text{ and } \forall e\in\Secinfty(E),
	\end{align*}
for a necessarily unique $X\in\Secinfty(TM)$.
\end{definition}

\begin{remark}\label{Rem: Diffop}
Derivations of a vector bundle form a subspace of the first order differential operators from $E$ to itself. Moreover, if the 
vector bundle has rank one then each differential operator is a derivation.  
\end{remark}

\begin{lemma}
Let $E\to M$ be a vector bundle. The derivations of of $E\to M$ are the sections of a Lie algebroid $DE\to M$. The 
Lie bracket of $\Secinfty(DE)$ is the commutator. An element in the fiber $D_p E$ is a map  
$\delta_p\colon\Secinfty(E)\to E_P$, such that 
 \begin{align*}
	\delta_p(fe)=v_p(f)e(p)+f(p)\delta(e) \ \forall f\in\Cinfty(M)\ \text{ and } \forall e\in\Secinfty(E),
	\end{align*}
for a necessarily unique $v_p\in T_p M$. The anchor, or \emph{symbol}, $\sigma\colon DE\to TM$ the assignment 
$\sigma(\delta_p)=v_p$.   	
Moreover,  $DE\to M$ is called  Atiyah algebroid of $E\to M$.  
\end{lemma}
The Atiyah algebroid fits into the following exact sequence of Lie algebroids, which is called the Spencer sequence,  
	\begin{align*}
	0\to \End(E)\to DE \to TM \to 0,
	\end{align*}
where the Lie algebroid structure of $\End(E)$ is the (pointwise) commutator and the trivial anchor. The first arrow is given 
by the inclusion.
With this we see immediatly that 
$\rank(DE)=\rank(E)^2+\dim(M)$.

We assign a vector bundle $E\to M$ to a Lie algebroid $DE\to M$. A natural question is: is 
this assignment functorial? In fact it is not, unless we restrict the category of vector bundles by just taking regular 
vector bundle morphisms, i.e. vector bundle morphisms which are fiber-wise invertible. Note that a regular vector bundle 
morphism  $\Phi\colon E\to E'$ covering  a smooth map $\phi\colon M\to M'$ allows us to define the pull-back $\Phi^*(e')$ of 
a section $e'\in\Secinfty(E')$ by putting
	\begin{align*}
	(\Phi^* e')(p)=\Phi_p^{-1}(e'(\phi(p)).
	\end{align*}

\begin{lemma}
Let $E\to M$ and $E'\to M'$ be vector bundles and let  $\Phi\colon E\to E'$ be  a regular vector bundle morphism covering 
a smooth map
$\phi\colon M\to M'$. Then the map $D\Phi\colon DE\to DE'$ defined by 
	\begin{align*}
	D\Phi(\delta_p)(e')= \Phi_p(\delta_p(\Phi^*e'))
	\end{align*}	
is a Lie algebroid morphism covering $\phi$. Moreover, this assignment makes $D$ a 
functor from the category of vector bundles with regular vector bundle  morphisms to the category of Lie algebroids.   
\end{lemma}

Note that for a given vector bundle $E\to M$, $DE$ has a tautological representation on $E$. Hence, we can define its de 
Rham complex with coefficients in $E$.

\begin{definition}     
Let $E\to M$ be a vector bundle. The complex $(\Omega_E^\bullet(M):=\Secinfty(\Anti^\bullet (DE)^*\otimes E), \D_E)$, where we define 
	\begin{align*}
	\D_E\alpha(\Delta_0,\dots, \Delta_k)&
	=\sum_{i=0}^k (-1)^i \Delta_i(\alpha(\Delta_0,\dots,\widehat{\Delta_i},\dots,\Delta_k)\\&
	+\sum_{i<j}(-1)^{i+j}\alpha([\Delta_i,\Delta_j],\Delta_0,\dots,
	\widehat{\Delta_i},\dots,\widehat{\Delta_j},\dots,\Delta_k)
	\end{align*}
for $\alpha\in \Omega_E^k(M)$  and $\Delta_i\in \Secinfty(DE)$, is called the $Der$-complex of 
$E$. We call elements of this complex  \emph{Atiyah forms}.
\end{definition}

\begin{lemma}
Let $E\to M$ be a vector bundle. The $Der$-complex $(\Omega_E^\bullet(M), \D_E)$ is acyclic with contracting 
homotopy $\iota_\mathbbm{1}$, the contraction with the identity operator $\mathbbm{1}\in \Secinfty(DE)$. 
\end{lemma}

\begin{remark}\label{Rem: Pull-back}
For a regular vector bundle morphism $\Phi\colon E \to E'$ covering the smooth map $\phi\colon M\to M'$, we have a pull-back
for Atiyah forms :
	\begin{align*}
	(\Phi^*\alpha')(\Delta_1,\dots, \Delta_k):= \Phi_p^{-1}
	 \big(\alpha'_{\phi(p)}(D\Phi(\Delta_1),\dots, D\Phi(\Delta_k))\big)
	\end{align*}
for $\Delta_i\in D_p E$, $\alpha\in \Omega_{E'}(M')$ and $p\in M$. This pull-back commutes with the differential $\D_E$.
\end{remark}

\begin{remark}
In view of Remark \ref{Rem: Diffop},  we have  $J^1L=(DL)^*\otimes L$, where $J^1 L$ 
is first jet bundle, i.e. $\Omega^1_L(M)=\Secinfty(J^1L)$. 
\end{remark}

\begin{remark}[Contact Geometry reloaded] A contact manifold is a pair $(M, \xi)$ where $M$ is a manifold and 
$\xi\subset TM$ is a co-dimension one maximally non-integrable distribution. Let us define the \emph{contact 1-form} 
$\Theta\colon TM\to L$ for $L:=TM/\xi$ simply as the projection. Now we consider 
	\begin{align*}
	\Theta \circ \sigma\colon DL \to L, 
	\end{align*}
which is Atiyah 1-form and $\omega=\D_L\sigma^*\Theta$. We call $\omega$ a 
contact 2-form. The maximal non-integrability of $\xi$ is equivalent to  the invertibility of 
	\begin{align*}
	\omega^\flat \colon DL\to (DL)^*\otimes L.
	\end{align*} 	
Since the $Der$-complex is acyclic the data of a line bundle and a non-degenerate and closed Atiyah 2-form in the 
$Der$-complex is actually equivalent to a contact structure.  From now on, we will always take this point of view on contact 
structure. So, in the sequel, a contact structure will be just a contact $2$-form i.e. a closed and non-degenerate Atiyah 2-
form on a line bundle.
Using  Remark \ref{Rem: Pull-back}, we can define \emph{contactomorphisms}:
Let $(L\to M,\omega )$ and $(L'\to M',\omega')$ be two contact manifolds and let  $\Phi\colon L\to L'$ be a regular line 
bundle morphism. We call $\Phi$ a contactomorphism, if $\Phi^*\omega'=\omega$.
This approach has several advantages, which will be clearer later on.
\end{remark}

We discuss now briefly Jacobi brackets
\begin{definition} 
 A Jacobi bracket on a line bundle $L\to M$ is a Lie bracket $\{-,-\}\colon \Secinfty(L)\times \Secinfty(L)\to \Secinfty(L)$, such that the bracket is a derivation in each slot.
\end{definition}

\begin{remark}
Let $\{-,-\}$ be a Jacobi bracket on a line bundle $L\to M$. Then  there is a unique tensor, called the Jacobi tensor, $J\in 
\Secinfty(\Anti^2(J^1L)^*\otimes L )$, such that 	
	\begin{align*}
	\{\lambda,\mu\}=J(j^1\lambda, j^1 \mu)
	\end{align*}	 
for $\lambda,\mu\in\Secinfty(L)$. Conversely, every $L$-valued $2$-form $J$ on $J^1 L$ defines a skew-symmetric bilinear bracket $\{-,-\}$, but the latter needs not to be a Jacobi bracket. Specifically, it does not need to fulfill the Jacobi identity. However, there is the  notion of a Gerstenhaber-Jacobi bracket 
	\begin{align*}
	[-,-]\colon \Secinfty(\Anti^i(J^1L)^*\otimes L )\times \Secinfty(\Anti^j(J^1L)^*\otimes L )\to
	\Secinfty(\Anti^{i+j-1}(J^1L)^*\otimes L ), 
	\end{align*}
such that the Jacobi identity of $\{-,-\}$ is equivalent to $[J,J]=0$	
see \cite[Chapter 1.3]{2017arXiv170508962T} for a detailed discussion. Finally, a Jacobi tensor defines  a map 
$J^\sharp\colon J^1L\to (J^1L)^*\otimes L=DL$.   
\end{remark}

\begin{lemma}
Let $L\to M$ be a line bundle and let $\omega\in \Secinfty(\Anti^2 (DL)^*\otimes L)$ be a contact 2-form. 
The inverse of the map 
	\begin{align*}
	\omega^\flat\colon DL\to (DL)^*\otimes L=J^1L 
	\end{align*}
is the sharp map of a tensor $J\in \Secinfty(\Anti^2 (J^1 L)^*\otimes L)$, i.e. 
	\begin{align*}
	(\omega^\flat)^{-1} =J^\sharp\colon J^1 L\to (J^1L)^*\otimes L= DL, 
	\end{align*}
for some $J\in \Secinfty(\Anti^2(J^1L)^*\otimes L)$.
Moreover, $J$ is a Jacobi tensor and  we will refer to it as the Jacobi tensor corresponding to the contact 2-form $\omega$.
\end{lemma}

When $L$ is the trivial line bundle, than the notion of Jacobi bracket boils down to that of \emph{Jacobi pair}.

\begin{remark}[Trivial Line bundle]\label{Rem: TrivLine}
Let $\mathbb{R}_M\to M$ be the trivial line bundle and let 
$J$ be a Jacobi tensor on it. 
Let us denote by $1_M\in \Secinfty(\mathbb{R}_M)$ the canonical global section. 
Using the canonical connection 
	\begin{align*}
	\nabla\colon TM\ni v \mapsto (f\cdot1_M\mapsto v(f)1_M)\in D\mathbb{R}_M,
	\end{align*}
we can see that $DL\cong TM\oplus \mathbb{R}_M$ and hence 
	\begin{align*}
	J^1\mathbb{R}_M=(D\mathbb{R}_M)^*\otimes \mathbb{R}_M= T^*M\oplus \mathbb{R}_M.
	\end{align*}	 
With this splitting, we see that 
	\begin{align*}
	J=\Lambda + \mathbbm{1}\wedge E
	\end{align*}	 	
for some $(\Lambda,E)\in\Secinfty(\Anti^2 TM\oplus TM)$. The Jacobi identity is equivalent to $[\Lambda,\Lambda]+E\wedge
\Lambda=0$ and $\Lie_E\Lambda=0$. The pair $(\Lambda,E)$ is often referred to as \emph{Jacobi pair}.
  Moreover, if we denote by $
\mathbbm{1}^*\in \Secinfty(J^1\mathbb{R}_M)$ the canonical 
section then we can write any 
$\psi\in J^1\mathbb{R}_M$ as $\psi=\alpha+r\mathbbm{1}^*\in \Secinfty(J^1\mathbb{R}_M) $, 
for some $\alpha\in T^*M$ and $r\in \mathbb{R}$. We obtain 
	\begin{align*}
	J^\sharp(\alpha+r\mathbbm{1}^*)=\Lambda^\sharp(\alpha)+r E-\alpha(E)\mathbbm{1}.	
	\end{align*}	
A more detailed discussion about Jacobi structures on trivial line bundles can be found in
 \cite[Chapter 2]{2017arXiv170508962T}
\end{remark}

\subsection{Generalized Geometry in odd dimensions}
In this section we briefly discuss generalized geometry in odd dimensions. This  introduction is far from 
being complete. For a more detailed outlook to the topic we refer the reader to \cite{DirJacBun} and \cite{GenConBun}.
The arena for  generalized geometry in odd dimensions is the so-called \emph{omni-Lie algebroid}.

\begin{definition}
Let $L\to M$ be a line bundle and let $H\in\Omega_L^3(M)$ be a closed Atiyah 2-form. The vector bundle $\mathbb{D}
L:=DL\oplus J^1 L$ together with 
	\begin{enumerate}
	\item the (Dorfman-like, H-twisted) bracket on sections
		\begin{align*}
		[\![(\Delta_1,\psi_1) ,(\Delta_2,\psi_2 )]\!]_H
		=([\Delta_1,\Delta_2],\Lie_{\Delta_1} \psi_2- \iota_{\Delta_2}\D_L\psi_1+\iota_{\Delta_1}\iota_{\Delta_2}H)
		\end{align*}		 
	\item the non-degenerate $L$-valued pairing 
		\begin{align*}
		\bla (\Delta_1,\psi_1) ,(\Delta_2,\psi_2 )\bra := \psi_1(\Delta_2)+\psi_2(\Delta_1)
		\end{align*}
	\item the canonical projection $\pr_D\colon \mathbb{D}L\to DL$ 
	\end{enumerate}	
is called the (H-twisted) omni-Lie algebroid of $L\to M$.  
\end{definition}

As in even dimensions, the objects of interest are certain subbundles of the Omni-Lie algebroid, or in our case its 
complexification.

\begin{definition}
Let $L\to M$ be a line bundle and let $H\in\Omega_L^3(M)$ be closed. A subbundle 
$\mathcal{L}\subseteq\mathbb{D}_\mathbb{C} L$ is called \emph{almost generalized contact}, 
if the following conditions are
fulfilled:
	\begin{enumerate}
	\item $\mathcal{L}$ is maximally isotropic with respect to (the $\mathbb{C}$-linear extension of) the  pairing 
	$\bla-,-\bra$.
	\item $\mathcal{L}\cap\cc{\mathcal{L}}=\{0\}$
	\end{enumerate}
If $\mathcal{L}$ is additionally involutive with respect to (the $\mathbb{C}$-linear extension of)  $[\![- ,-]\!]_H$, we
call $\mathcal{L}$ an $H$-generalized contact structure, if  $H=0$ we say just generalized contact structure.
Moreover, an $H$-generalized contact bundle is a line bundle together with a $H$-generalized contact structure.
\end{definition}

\begin{remark}
Note that so far, in  the literature only the case $H=0$ considered. The main motivation of introducing a non-trivial $H$ is, that if one considers just  a short exact sequence of vector bundles
	\begin{center}
		\begin{tikzcd}
 		0 \arrow[r]&  J^1 L   \arrow[r]& E\arrow[r] &  DL \arrow[r] \arrow[l, bend left, "\nabla^L"] & 0 \\
		\end{tikzcd}
	\end{center}
for a line bundle $L\to M$ and a splitting $\nabla^L\colon DL\to E$, then the closed Atiyah 3-form is the curvature of this splitting. 
\end{remark}

We can rephrase the definition in terms of an endomorphism of the omni-Lie algebroid, which shows the similarity to generalized complex geometry:

\begin{proposition}\label{Prop: GenConEnd}
Let $L\to M$ be a line bundle and let $H\in\Omega_L^3(M)$ be closed. For a $\mathbb{K}\in 
\Secinfty(\End\mathbb{D}L)$, fulfilling 
	\begin{enumerate}
	\item $\mathbb{K}^2=-\id$
	\item $\bla\mathbb{K}-,\mathbb{K}-\bra =\bla-,-\bra$
	\end{enumerate}
the $+i$-Eigenbundle is an almost generalized contact structure.
Moreover, every almost generalized contact structure $\mathcal L$ arises as the $+\I$-Eigenbundle of an endomorphism $\mathbb{K}$ satisfying i.)-ii.). If additionally 
	\begin{align*}
	\mathcal{N}^H_\mathbb{K}(A,B)
	:=[\![\mathbb{K}(A) ,\mathbb{K}(B)]\!]_H-\mathbb{K}([\![\mathbb{K}(A),B]\!]_H)-\mathbb{K}(
	[\![A,\mathbb{K}(B)]\!]_H)-[\![A,B]\!]_H=0 
	\end{align*}	  
holds for all $A,B\in\Secinfty(\mathbb{D}L)$, then then $\mathcal{L}$ is $H$-generalized contact. 
\end{proposition}

\begin{remark}
Using Proposition \ref{Prop: GenConEnd}, we will often call the endomorphism $\mathbb{K}$ itself an $H$-generalized 
contact structure. If just the conditions \textit{i.)} and \textit{ii.)} are fulfilled,  
we will refer to $\mathbb{K}$ as an almost generalized contact 
structure.
\end{remark}

\begin{lemma}
Let $L\to M$ be a line bundle, let $H\in\Omega_L^3(M)$ be closed and let  $\mathbb{K}\in 
\Secinfty(\End\mathbb{D}L)$ be an almost generalized contact structure. Then, using the decomposition
$\mathbb{D}L=DL\oplus J^1 L$,  $\mathbb{K}$ can be written as 
	\begin{align*}
	\mathbb{K}=
	\begin{pmatrix}
	\phi & J^\sharp \\
	\sigma^\flat & -\phi^*\\
	\end{pmatrix},
	\end{align*}
where $\phi\in\Secinfty(\End DL)$, $J\in\Secinfty(\Anti^2 (J^1 L)^*\otimes L)$ and 
$\sigma\in\Omega_L^3(M)$. If $\mathbb{K}$ is additionally  $H$-generalized contact, then $J$ is a Jacobi 
tensor.
\end{lemma}
\begin{proof}
The proof is an easy verification, but can also be found in \cite{2017arXiv171108310S} or \cite{GenConBun}.
\end{proof}

\begin{example}
Let $L\to M$ be a line bundle and let $\omega\in \Omega_L^2(M)$ be a contact 2-form. Then 
	\begin{align*}
	\mathcal{L}=\{ (\Delta,\I\iota_{\Delta}\omega)\in\mathbb{D}_\mathbb{C}L\ | \ \Delta\in D_\mathbb{C}L\}
	\end{align*}
is a generalized contact structure. The corresponding endomorphism $\mathbb{K}\in \Secinfty(\End\mathbb{D}L)$ is given by 
	\begin{align*}
	\mathbb{K}=
	\begin{pmatrix}
	0 & J^\sharp \\
	-\omega^\flat & 0\\
	\end{pmatrix},
	\end{align*}
where $J$ is the Jacobi tensor of $\omega$.
\end{example}

\begin{example}
Let $L\to M$ be a line bundle and let $\phi\in\Secinfty(\End DL)$ be a complex structure, i.e. an almost complex structure 
such 
that the Nijenhuis torsion with respect to the Lie algebroid bracket vanishes. If we denote by $DL^{(1,0)}$ its 
$+\I$-Eigenbundle, then
	\begin{align*}
	\mathcal{L}=DL^{(1,0)}\oplus \mathrm{Ann}(DL^{(1,0)})
	\end{align*}	   
is a generalized contact structure with corresponding endomorphism $\mathbb{K}\in \Secinfty(\End\mathbb{D}L)$ given by
	\begin{align*}
	\mathbb{K}=
	\begin{pmatrix}
	\phi & 0\\
	0 & -\phi^*\\
	\end{pmatrix},
	\end{align*}
where $\phi^*\in \End(J^1L)$ is the adjoint of $\phi$ with respect to the $L$-valued pairing of $J^1L$ and $DL$. In the following, we will refer to $\phi$ as a \emph{Gauge complex} or \emph{Atiyah complex} structure on $L$.
\end{example}
We will refer to the last two examples, i.e. contact structures and Atiyah complex structures,  as the extreme cases of generalized contact structures.

We will not discuss the automorphisms of $\mathbb{D}L$ in detail. A conceptual discussion can be found in  
\cite{2017arXiv171108310S}. Nevertheless, we will need a kind of action of Atiyah 2-forms on $H$-generalized contact structures: 
\begin{lemma}\label{Lem: TransGenCon}
Let $L\to M$ be a line bundle, let $H\in\Omega_L^3(M)$ be closed and let $\mathcal{L}$ be a 
$H$-generalized contact structure. For a real $B\in\Omega_L^2(M)$, the subbundle 
	\begin{align*}
	\mathcal{L}^B=\{(\Delta,\psi+\iota_\Delta B)\in \mathbb{D}L\ | \  (\Delta,\psi)\in \mathcal{L} \}
	\end{align*}
is a $(H+\D_L B)$-generalized contact structure.
If $B$ is closed we will refer to it as a \emph{B-field}. 
\end{lemma}

\begin{remark}\label{Rem:H-gencon}
Let $L\to M$ be a line bundle, let $H\in\Omega_L^3(M)$ be closed and let $\mathcal{L}$ be a 
$H$-generalized contact structure. Since the $Der$-complex is acyclic with contracting homotopy 
$\iota_\mathbbm{1}$, we have that $\mathcal{L}^{-\iota_{\mathbbm{1}}H}$ is a generalized contact structure by the previous Lemma.  
\end{remark}

\section{Tranversally Complex Jacobi Structures and Generalized Contact Bundles}\label{Sec: WeakJac}
Unlike in Poisson geometry, a Jacobi structure $J$ may also have odd dimensional characteristic leaves. This comes from the fact that the
leaves are the integral manifolds of the singular distribution $\image(\sigma\circ J^\sharp)$, where the image of the 
Jacobi tensor $J$ is surely an even dimensional subbundle of $DL$, but composed with symbol we might lose one dimension. It 
seems therefore reasonable to distinguish between regular Jacobi structures, i.e. Jacobi structure with a regular 
distribution, and    

\begin{definition}
Let $L\to M$ be a line bundle and let $J\in \Secinfty(\Anti^2 (J^1L)^*\otimes L)$ be a Jacobi tensor. J is said to be 
weakly regular, if $\image(J^\sharp)\subseteq DL$ is a regular subbundle. 
\end{definition}

\begin{remark}
A Jacobi structure which is weakly regular is not always regular. To illustrate this, we take for example the canonical Jacobi structure 
$(\Lambda_{can}, E_{can})\in \Secinfty(\Anti^2T\mathbb{R}^{2k+1}\oplus T\mathbb{R}^{2k+1})$ coming from the contact 
structure and consider $Z\in\Secinfty(T\mathbb{R})$ given by $Z=x\frac{\partial}{\partial x}$. Then 
$(\Lambda=\Lambda_{can}+E_{can}\wedge Z,E_{can})$ defines a weakly regular Jacobi structure on
$\mathbb{R}^{2k+1}\times \mathbb{R}$ where the set of contact points are $\{ (x,0)\in \mathbb{R}^{2k+1}\times \mathbb{R}\}$. 
\end{remark}

\begin{remark}
Let $L\to M$ be a line bundle and let $J\in \Secinfty(\Anti^2 (J^1 L)^*\otimes L)$ be weakly regular, then by definition 
$\image(J^\sharp)\subseteq DL$ 
is a regular subbundle. Moreover, one can prove, that it is in fact a subalgebroid and there is a canonical form 
$\omega\in \Secinfty(\Anti^2(\image(J^\sharp))\otimes L)$ such that $\D_{\image(J^\sharp)}\omega=0$ and $\omega(J^\sharp(\alpha),J^\sharp(\beta))=\alpha(J^\sharp(\beta))$, where $\D_{\image(J^\sharp)}$ the de Rham differential with coefficients in the tautological representation $\image(J^\sharp)\to DL$. We will refer to this form as the inverse of $J$ and denote it by $J^{-1}$.  
\end{remark}

A weakly regular Jacobi structure alone is not enough to construct an almost generalized contact structure out of, more precisely we need to consider transversal information to be seen in the following 

\begin{definition}\label{Def: TransCom}
A transversally complex subbundle on $L\to M $ is a pair $(S,K)$ consisting of two involutive subbundles
$S\subseteq DL$ and $K\subseteq D_\mathbb{C}L$ , such that 
	\begin{enumerate}
	\item $K+\cc{K}=D_\mathbb{C}L$,
	\item $K\cap\cc{K}=S_\mathbb{C}$.
	\end{enumerate}	 
 
\end{definition}

\begin{remark}
The name  \emph{transversally complex involutive subbundle} comes from the fact, that the decomposition 
	\begin{align*}
	\big(\faktor{DL}{S}\big)_\mathbb{C}=\big(\faktor{K}{S_\mathbb{C}}\big)\oplus \big(\faktor{\cc{K}}{S_\mathbb{C}}\big)
	\end{align*}
defines an almost complex structure on $\faktor{DL}{S}$.  
\end{remark}
We are mainly interested in transversally complex structures with an additional Jacobi structure, let us therefore be precise in the following  

\begin{definition}
Let $L\to M$ be a line bundle. A transversally complex Jacobi structure is a pair $(J,K)$ consisting of a weakly regular 
Jacobi structure $J\in \Secinfty(\Anti^2 (J^1L)^*\otimes L)$ and an involutive subbundle $K\subset D_\mathbb{C}L$, such that 
$(\image(J^\sharp),K)$ is transversally complex subbundle.
\end{definition} 

This kind of structure appear naturally in generalized contact geometry if one assumes some regularity conditions, to be seen in the next

\begin{proposition}
Let $L\to M$ be a line bundle and let $\mathcal{L}\subseteq \mathbb{D}_\mathbb{C}L$ be a generalized contact structure, such 
that the corresponding Jacobi structure $J$ is weakly regular. Then $(J,\pr_D\mathcal{L})$ is a transversally complex Jacobi structure.
\end{proposition}

	\begin{proof}
	Let us define $K:=\pr_D(\mathcal{L})$.
	Having in mind that $\image(J^\sharp)_\mathbb{C} =\pr_D(\mathcal{L})\cap\pr_D(\cc{\mathcal{L}})$ and that $	
	\pr_D(\mathcal{L})$ is involutive
	(\cite{2017arXiv171108310S}), we get the result. 
	  
	\end{proof}

With the previous proposition in mind, it is natural to ask which transversally complex Jacobi structure can be induced by a generalized contact structure. To formalize the term "\emph{induced by}", we use the proof of the previous Lemma.

\begin{definition}
Let $L\to M$ be a line bundle, let $\mathcal{L}\subseteq\mathbb{D}_\mathbb{C} L$ be a a generalized contact structure, and 
$(J,K)$ be a transversally complex Jacobi structure. We say $\mathcal{L} $
induces $(J,K)$, if J is the Jacobi structure of $\mathcal{L}$ and $K=\pr_D\mathcal{L}$.
\end{definition}

Let us give a first characterization of a transversally complex Jacobi structure induced by a generalized contact structure.

\begin{lemma}
Let $L\to M$ be a line bundle and let $(J,K)$ be 
transversally complex Jacobi structure induced by the generalized contact structure 
$\mathcal{L}\subseteq\mathbb{D}_\mathbb{C}L$. Then 
\begin{enumerate}
 \item for any real extension $\omega$ of the inverse of $J$ there exists a real $B\in \Omega_L^2(M)$, such that 
 	\begin{align*}
 	\D_L(\I\omega+B)(\Delta_1,\Delta_2,\Delta_3)=0\ \forall\Delta_i\in K.
 	\end{align*}
 \item $\mathcal{L}=(K\oplus\mathrm{Ann}(K))^{\I\omega+B}$
\end{enumerate}
\end{lemma}

\begin{proof}
Every generalized contact structure $\mathcal{L}$ can be represented by a two form $\tilde\varepsilon \colon \Anti^2\pr_D
\mathcal{L}\to L_\mathbb{C}$ by
	\begin{align*}
	\mathcal{L}=\{(\Delta,\alpha)\in\mathbb{D}_\mathbb{C}L\ | \ \alpha\at{\pr_D\mathcal{L}}=
	\iota_\Delta\tilde\varepsilon\at{\pr_D\mathcal{L}}\},
	\end{align*}
such that $\mathrm{Im}(\tilde\varepsilon)\at{\Anti^2 S}= (J\at{S})^{-1}$ for the Jacobi structure $J$ of 
the generalized contact 
structure, so in our case with the given weakly regular one. 
The proof of this can be found in \cite[Section 2.2.3]{2017arXiv171108310S}.
If the generalized contact structure induces the given transversally complex Jacobi structure, the we have $\pr_D(\mathcal{L})=K$.
Since $K$ is regular, we extend $\tilde\varepsilon$ to a (complex valued) 2-form $\varepsilon$ and get 
	\begin{align*}
	\mathcal{L}&=(\pr_D(\mathcal{L})\oplus\mathrm{Ann}(\pr_D(\mathcal{L})))^\varepsilon\\&
	=(K\oplus\mathrm{Ann}(K))^{\mathrm{Re}(\varepsilon)+\I\mathrm{Im}(\varepsilon)},
	\end{align*}
which is the first statement, since $\mathrm{Im}(\varepsilon)$ extends the inverse of $J$. The second statement follows 
directly from the integrability of $\mathcal{L}$. 
\end{proof}

\begin{remark}\label{Rem: Equivalence H-genCon}
It is easy to see, that if a generalized contact structure induces a given transversally complex Jacobi structure,
then every $B$-
field transform of it will induce the same structure. So in view of Remark \ref{Rem:H-gencon}, we can even say that a 
transversally 
complex Jacobi structure is induced by a $H$-generalized contact structure, if and 
only if it is induced by a generalized contact structure. Note that this is not the case in generalized complex geometry,
since the third de Rham cohomology need not to be zero, while the Der-complex is always acyclic.
\end{remark}

As an endpoint of this section we collect all the previous results in the following
\begin{corollary}
Let $L\to M$ be a line bundle and let $(J,K)$ be a transversally complex Jacobi structure. These data come from a
($H$-)generalized 
contact structure, if and only if there exists a real extension $\omega$ of the inverse of $J$ and a real $B\in \Omega_L^2(M)
$, such that
	\begin{align*}
	\D_L(\I\omega+B)(\Delta_1,\Delta_2,\Delta_3)=0\ \forall\Delta_i\in K.
	\end{align*}	 
\end{corollary}

This condition seems not to be very easy to handle and also involves an extension of the inverse of the Jacobi structure and  the existence of a 2-form $B$. We will see in the following that their existence can be encoded completely in properties of $J$.  

\section{The Splitting of the Spectral Sequence of a Transversally Complex Subbundle}

We have seen that every generalized contact structure with weakly regular Jacobi structure induces a transversally complex 
Jacobi structure. The latter special case of a transversally complex subbundles. 
We want to explore these subbundles by means of a canonical spectral sequence attached to them. It turns out that a 
transversally complex Jacobi structure comes from a generalized contact bundle, if 
and only if a certain cohomology class in the first page of this spectral sequence vanishes.

Throughout this subsection we will assume the following data: a line bundle $L\to M$ and the subalgebroids 
$S\subseteq DL$ 
and $K\subseteq D_\mathbb{C}L$, such that $K+\cc{K}=D_\mathbb{C}L$ and $K\cap \cc{K}=S_\mathbb{C}$, in other words we want to fix a transversally complex subbundle $(S,K)$. Moreover, if not stated otherwise, we see every Atiyah form as complex valued.

\subsection{General Statements and Preliminaries}\label{SubSec: SpecSeq}
This part is not only meant to fix notation and give a quick reminder on spectral sequences, but also to give a splitting of the zeroth and first page of the spectral sequence induced by the transversally complex subbundle $(S,K)$. Let us begin by showing that $(S, K)$ induces two filtered complexes.

\begin{lemma}\label{Lem: filtCom}
The subspaces 
	\begin{align*}
	&F_n^m:=\{ \alpha\in \Omega_L^m(M)\ | \ \iota_X\alpha=0 \ \forall X\in \Anti^{m-n+1}S_\mathbb{C}\} \text{ and }\\&
	G_n^m:=\{ \alpha\in \Omega_L^m(M)\ | \ \iota_X\alpha=0 \ \forall X\in \Anti^{m-n+1}K\}
	\end{align*}
fulfill the following properties
	\begin{enumerate}
	\item $\Omega_L^m(M)=F_0^m\supseteq F_1^m\supseteq\dots$ and $\Omega_L^m(M)=G_0^m\supseteq G_1^m\supseteq\dots$
	\item $\D_L(F_n^m)\subseteq F_n^{m+1}$ and $\D_L(G_n^m)\subseteq G_n^{m+1}$
	\end{enumerate}
Moreover, we have the following relations between them:
	\begin{enumerate}[(I)]
	\item $G_n^m\subseteq F_n^m$
	\item $G_j^m\cap \cc{G}_i^m \subseteq F_{i+j}^m$
	\item $\langle G_j^m\cap \cc{G}_i^m\rangle_{i+j=n}=F_n^m$
	\item $G_j^m\cap \cc{G}_i^m\cap(\langle G_l^m\cap \cc{G}_k^m\rangle_{i+j=n,(i,j)\neq(k,l)}+F_{i+j+1}^m)
	\subseteq F_{i+j+1}^m$
	\item $F_{i+j+n}^m\cap G_j^m\cap \cc{G}_i^m=\langle G_{j+k}^m\cap \cc{G}_{i+l}^m\rangle_{k+l=n}$
	\end{enumerate}
\end{lemma}

\begin{proof}
The proof is an easy verification exploiting the involutivity of $S$ and $K$ and the relations  $K+\cc{K}=D_\mathbb{C}L$ and
$K\cap \cc{K}=S_\mathbb{C}$.
\end{proof}

The properties \textit{i.)} and \textit{ii.)} in the previous Lemma show that the subspaces $F^m_n$ and $G_n^m$ induce  filtrations of the 
Der-complex. 
Note that we did not explicitly introduce the spaces $\cc{G}_i^m$, but from the notation it should be clear that we mean the complex conjuagtion of the spaces $G_i^m$ or equivalently the filtered complex with respect to $\cc{K}$.
Properties \textit{(I)-(V)} will give us a canonical splitting of the spectral sequence and its differentials. 

Let us therefore briefly recall the definition of a spectral sequence
\begin{definition}
A spectral sequence is a series of bigraded vector spaces $\{E^{\bullet,\bullet}_r\}_{r\geq 0}$, called the pages, and a series of maps 
$\{\D^r\colon E^{\bullet,\bullet}_r\to E^{\bullet+r,\bullet+1-r}_r \}_{r\geq 0}$, the differentials,  such that 
	\begin{enumerate}
	\item $(\D^r)^2=0$
	\item $E^{\bullet,\bullet}_r\cong 
	\frac{\ker(\D^{r-1}\colon E^{\bullet,\bullet}_{r-1}\to E^{\bullet+r-1,\bullet-r}_{r-1} )}
	{\image(\D^{r-1}\colon E^{\bullet-r+1,\bullet+r}_{r-1}\to E^{\bullet,\bullet}_{r-1})}$
	\end{enumerate}
\end{definition}

There is a canonical way to associate a spectral sequence to a filtered complex. We will define it for the filtered complex 
$\Omega_L^m(M)=F_0^m\supseteq F_1^m\supseteq\dots$. We consider the quotients 
	\begin{align*}
	E^{p,q}_r=\frac{\{\alpha\in F_p^{q+p}\ | \ \D_L\alpha \in F_{p+r}^{q+p+1}\}}{F_{p+1}^{q+p}+\D_L(F_{p+1-r}^{q+p-1})}
	\end{align*}
together with the maps 
	\begin{align}\label{Eq: Differential}
	\D^r\colon E^{p,q}_r\ni \alpha +F_{p+1}^{q+p}+ \D_L(F_{p+1-r})\mapsto
	\D_L\alpha+ F_{p+r+1}^{q+p+1} + \D_L(F_{p+1}^{q+p+1})\in E^{p+r,q+1-r}_r.
	\end{align}

\begin{lemma}
The maps $\{\D^r\}_{r\geq0}$ from Equation \ref{Eq: Differential} are well-defined and $\{(E^{\bullet,\bullet}_r,\D^r)\}_{r\geq 0}$ is a spectral sequence.
\end{lemma}

\begin{proof}
The proof is a easy exercise, but can be found in every book treating spectral sequences of filtered complexes, see e.g. \cite{weibel1995introduction}.
\end{proof}
In our case, we do not have only one filtered complex, but two more filtered complexes $\Omega_L^m(M)=G_0^m\supseteq F_1^m
\supseteq\dots$ and its complex conjugate, which are in relation with $\Omega_L^m(M)=F_0^m\supseteq F_1^m\supseteq\dots$. Let us now consider the 
quotients 
	\begin{align*}
	E^{(i,j),q}_r=
	\frac{\{\alpha\in G_j^{q+i+j}\cap \cc{G}_i^{q+i+j}\ | \ \D_L\alpha \in F_{i+j+r}^{q+i+j+1}\}}
	{(F_{i+j+1}^{q+i+j}+\D_L(F_{i+j+1-r}^{q+i+j-1}))\cap  G_j^{q+i+j}\cap \cc{G}_i^{q+i+j}}
	\end{align*}

\begin{lemma}
Let $ s=0,1$,  then
the canonical maps 
	\begin{align*}
	E_s^{(i,j),q}\ni \alpha + (F_{i+j+1}^{q+i+j}+\D_L(F_{i+j+1-s}))\cap  G_j^{q+i+j}\cap \cc{G}_i^{q+i+j}\mapsto
	\alpha + F_{i+j+1}^{q+i+j}+ \D_L(F_{i+j+1-s}^{q+i+j-1}) \in E_s^{i+j,q}
	\end{align*}
are injective and moreover $E_s^{p,q}=\bigoplus_{i+j=p}E_s^{(i,j),q}$.
\end{lemma}

\begin{proof}
Injectivity follows by defintion of the maps. Let us start with $s=0$. We have 
\begin{align*}
E^{(i,j),q}_0=\frac{G_j^{q+i+j}\cap \cc{G}_i^{q+i+j}}{F_{i+j+1}^{q+i+j}\cap  G_j^{q+i+j}\cap \cc{G}_i^{q+i+j}},
\end{align*}
where we used  property \textit{ii.)} of Lemma \ref{Lem: filtCom}.
By \textit{(III)} of Lemma \ref{Lem: filtCom}, we obtain that 
$E_0^{p,q}=\langle E_0^{(i,j),q}\rangle_{i+j=p}$. Let now $\omega_{ij}\in G_j^{q+i+j}\cap \cc{G}_i^{q+i+j}$ for $i+j=p$, such 
that 
	\begin{align*}
	\sum_{i+j=p}\omega_{ij}\in F^{q+i+j}_{i+j+1}.
	\end{align*}
We have that
\begin{align*}
\omega_{kl}\in G_l^{q+i+j}\cap \cc{G}_k^{q+i+j}\cap(\langle G_j^{q+i+j}\cap \cc{G}_i^{q+i+j}\rangle_{k+l=n,
(i,j)\neq(k,l)}+F_{i+j+1}^{q+i+j})\subseteq F_{i+j+1}^{q+i+j}
\end{align*}
 for every choice of $k+l=i+j$ and hence, using \textit{(IV)} 
of Lemma 
\ref{Lem: filtCom}, 
$\omega_{kl}\in F_{i+j+1}^{q+i+j}$. If we pass to the  quotients, we get the result for $s=0$. Let us pass to 
$s=1$ and let $\omega\in F_{i+j}^{q+i+j}$ such that $\D_L\omega\in  F_{i+j+1}^{q+i+j+1}$. Since  
$F_{i+j}^{q+i+j}=\langle \langle G_l^{q+i+j}\cap \cc{G}_k^{q+i+j}\rangle_{k+l=i+j}$, we can find 
$\omega_{kl}\in G_l^{q+i+j}\cap \cc{G}_k^{q+i+j}$ for $k+l=i+j$, such that 
	\begin{align*}
	\omega=\sum_{k+l=i+j} \omega_{kl}.
	\end{align*}
We have that $\D_L\omega_{kl}\in G_l^{q+i+j+1}\cap \cc{G}_k^{q+i+j+1}$ similarly as in  case $s=0$, we can prove that 
actually  $\D_L\omega_{kl}\in F_{i+j+1}^{q+i+j+1}$, using $\D_L\omega\in F_{i+j+1}^{q+i+j+1}$. Thus
\begin{align*}
\omega\in \langle
\{\alpha\in G_k^{q+i+j}\cap \cc{G}_l^{q+i+j}\ | \ \D_L\alpha \in F_{i+j+r}^{q+i+j+1}\}\rangle_{k+l=i+j}
\end{align*}
 and hence $E_1^{i
+j,q}=\langle E_1^{(k,l),q}\rangle_{k+l=i+j}$.  Let now $\omega_{kl}\in G_l^{q+i+j}\cap \cc{G}_k^{q+i+j}$ for $k+l=i+j$, 
such that $\D_L\omega_{kl}\in F_{i+j+1}^{q+i+j}$ and $\sum_{k+l=i+j}\omega_{kl}\in 
 F_{i+j+1}^{q+i+j} +\D_L(F_{i+j}^{q+i+j-1})$. Therefore there exists $\alpha\in F_{i+j}^{q+i+j-1} $, such that 
$\sum_{k+l=i+j}\omega_{kl} +\D_L\alpha \in F_{i+j+1}^{q+i+j}$. Splitting $\alpha=\sum_{k+l=i+j}\alpha_{kl}$ for some 
$\alpha_{kl}\in G_k^{q+i+j-1}\cap \cc{G}_l^{q+i+j-1}$ , we get that 
	\begin{align*}
	\sum_{k+l=i+j}\omega_{kl}+\D_L\alpha_{kl} \in F_{i+j+1}^{q+i+j}. 
	\end{align*}
Additionally we have $\omega_{kl}+\D_L\alpha_{kl}\in G_l^{q+i+j}\cap \cc{G}_k^{q+i+j} $. Applying the same argument as in the 
case $s=0$, we get $\omega_{kl}+\D_L\alpha_{kl}\in F_{i+j+1}^{q+i+j}$ for all $k+l=i+j$. Passing to the quotient, we get 
the result for $s=1$.
\end{proof}
We consider the differentials on the zeroth and first page and use this splitting to decompose them.

\begin{proposition}\label{Prop: SplittingDiff}
For the differentials $\D^0\colon E_0^{p,q}\to E^{p,q+1}_0$ and $\D^1\colon E^{p,q}_0\to E^{p+1,q}_0$, the following hold
	\begin{enumerate}
	\item $\D^0(E^{(i,j),q}_0)\subseteq E^{(i,j),q+1}_0$
	\item $\D^1(E^{(i,j),q}_1)\subseteq E^{(i+1,j),q}_1\oplus E^{(i,j+1),q}_1$
	\end{enumerate}
Hence there is a canonical splitting $\D^1=\partial^1+\cc{\partial}^1$, where 
$\partial^1(E_{(i,j),q})\subseteq E_{(i+1,j),q}$ and 
$\cc{\partial}^1(E_{(i,j),q})\subseteq E_{(i,j+1),q}$. Finally, $(\partial^1)^2=(\cc{\partial}^1)^2=
\partial^1\cc{\partial}^1+\cc{\partial}^1\partial^1=0$.
\end{proposition}

\begin{proof}
We start with the zeroth page. Let $\omega+F_{i+j+1}^{q+i+j}\in E_{(i,j),q}$, such that 
$\omega\in G_j^{q+i+j}\cap\cc{G}_i^{q+i+j}$, then 
	\begin{align*}
	\D^0(\omega+F_{i+j+1}^{q+i+j})= \D_L\omega+F_{i+j+1}^{q+i+j+1}.
	\end{align*}
We have that $\D_L\omega\in G_j^{q+i+j+1}\cap\cc{G}_i^{q+i+j+1}$ and hence $\D^0(\omega+F_{i+j+1}^{q+i+j})\in E_{(i,j),q+1}$.
For the first page let us choose $\omega+F_{i+j+1}^{q+i+j}+\D_L(F_{i+j}^{q+i+j-1})$, with 
$\omega\in G_j^{q+i+j}\cap\cc{G}_i^{q+i+j}$ and $\D_L\omega \in F_{i+j+1}^{q+i+j+1}$. Then
	\begin{align*}
	\D_L\omega\in G_j^{q+i+j+1}\cap\cc{G}_i^{q+i+j+1}\cap F_{i+j+1}^{q+i+j+1}=
	G_{j+1}^{q+i+j+1}\cap\cc{G}_i^{q+i+j+1}
	+ G_j^{q+i+j+1}\cap\cc{G}_{i+1}^{q+i+j+1}
	\end{align*}	  
and the claim follows.
\end{proof}

\begin{remark}
The whole spectral sequence (rather than just its zeroth and first pages) seem to be very interesting objects in themselves. However,  we will not explore it
here in more detail.
\end{remark}

\subsection{The Obstruction Class of transversally Complex Subalgebroids}\label{Sec: ObsTransComJac}
In Section \ref{Sec: WeakJac}, we have seen
that a transversally complex Jacobi structure $(J,K)$
comes from a generalized contact structure, if and only if 
there exists an extension of the inverse of $J$,  $\omega\in \Omega_L^2(M)$, and a real 2-form $B\in \Omega^2_L(M)$, such 
that
	\begin{align*}
	\D_L(\I\omega+B)(\Delta_1,\Delta_2,\Delta_3)=0\ \forall\Delta_i\in \pr_D\mathcal{L}.
	\end{align*}
We want to apply the techniques from the previous subsection to obtain a cohomological obstruction on this condition.
Using the formalism of Subsection \ref{SubSec: SpecSeq} and using the notation $\image(J^\sharp)=S$ , we see that this is 
equivalent to 	
	\begin{align*}
	\D_L(\I\omega+B)\in G_1^{3}=G_1^{3}\cap\cc{G}_0^3.
	\end{align*}
Since $\omega$ is non-degenerate on $S$, we have that $\omega\in F_0^2$ and $\omega\notin F_1^2$. Thus,
$\I\omega +B\in G_0^2\cap \cc{G}_0^2$ and $\D_L\omega , \D_L B\in F_1^3$. Hence both forms define classes in  
$\in E^{(0,0),2}_0$, denoted by $[\omega]_0$ and $[B]_0$. Moreover, the latter are both $\D^0$-closed in $E^{(\bullet,\bullet),\bullet}_0$
and hence they define iterated classes in $E^{(0,0),2}_1$,  denoted by $[[\omega]_0]_1$ and $[[B]_0]_1$. 

\begin{corollary}\label{Cor: EqutoCo}
The condition $\D_L(\I\omega+B)\in G_1^{3}=G_1^{3}\cap\cc{G}_0^3$ is equivalent to 
$\partial^1 (\I[[\omega]_0]_1+[[B]_0]_1)=0$.
\end{corollary} 

\begin{proof}
 We have that $\D_L(\I\omega+B)\in G_1^{3}\cap\cc{G}_0^3$, which implies for the class
	\begin{align*}
	\D^1(\I\omega+B+ F_{1}^{2}+\D_L(F_{0}^{1}))=
	\D_L(\I\omega+ B)+F_{2}^{3} +\D_L(F_{1}^{2})\in 
	G_1^{3}\cap G_0^{3}+F_{2}^{3} +\D_L(F_{1}^{2})
	\end{align*}
Hence $\D^1(\I[[\omega]_0]_1+[[B]_0]_1)\in E_{(0,1),2}^1$. 
Using the splitting of the differential $\D^1=\partial^1+\cc{\partial}^1$, we 
get that $\partial^1(\I[[\omega]_0]_1+[[B]_0]_1)=0$.
\end{proof}

We want to go a step further and ask for which $\omega$ can we find a $B$, such that  
$\D_L(\I\omega+B)\in G_1^{3}\cap\cc{G}_0^3$. The answer gives the following 

\begin{lemma}\label{Lem: FinObs}
Let $\omega\in \Omega_L^2(M)$ be real, such that $\D_L\omega\in F_1^{3}$. Then there exists a $B\in \Omega_L^2(M)$, such that
$\D_L(\I\omega+B)\in G_1^{3}=G_1^{3}\cap\cc{G}_0^3$ if and only if 
	\begin{enumerate}
	\item $\partial^1\cc{\partial}^1[[\omega]_0]_1=0$
	\item $\cc{\partial}^1[[\omega]_0]_1-\partial^1[[\omega]_0]_1$ is $\D^1$-exact in $E_{3}^{(\bullet,\bullet),\bullet}$.
	\end{enumerate}
\end{lemma} 

Before we prove the Lemma, we want to discuss \textit{i.)} and \textit{ii.)}. Note that \textit{ii.)} can only be fulfilled, if \textit{i.)}  is fulfilled, since \textit{i.)} just ensures that $\cc{\partial}^1[[\omega]_0]_1-\partial^1[[\omega]_0]_1$ is $\D^1$-closed. 
Let us now prove the Lemma.

\begin{proof}[of Lemma \ref{Lem: FinObs}]
Let us first assume, that $\D_L(\I\omega+B)\in G_1^{3}=G_1^{3}\cap\cc{G}_0^3$ for a real $B$, which is equivalent to 
$\partial^1(\I[[\omega]_0]_1+[[B]_0]_1)=0$ by Proposition \ref{Cor: EqutoCo}. Hence we have 
	\begin{align*}
	2\D^1[[B]_0]_1&
	=2\mathrm{Re}(\D^1[[\I\omega+B]_0]_1)\\&
	=2\mathrm{Re}(\cc{\partial}^1[[\I\omega+B]_0]_1)\\&
	=(\cc{\partial}^1[[\I\omega+B]_0]_1+\cc{\cc{\partial}^1[[\I\omega+B]_0]_1})\\&
	= \D^1[[B]_0]_1 +\I(\cc{\partial}^1[[\omega]_0]_1-\partial^1[[\omega]_0]_1),
	\end{align*}
and hence $\cc{\partial}^1[[\omega]_0]_1-\partial^1[[\omega]_0]_1$ is $\D^1$-exact. Assuming, on the other hand, that 
$\I(\cc{\partial}^1[[\omega]_0]_1-\partial^1[[\omega]_0]_1)=\D^1[B]_1$, where we can choose a real representant $B$ of $[[B]_0]_1$, then it is
easy to see that $\partial^1 [[\I\omega+B]_0]_1=0$ and hence the claim follows.
\end{proof} 
Let us conclude this section with the main theorem of this note, which is basically just a summary of the previous results. Afterwards we will discuss the connection to generalized complex structures. 

\begin{theorem}\label{Thm: ObsCls}
Let $L\to M$ be a line bundle and let $(J,K)$ be a transversally complex Jacobi structure on $L$.  These  
data are induced by a generalized contact structure, if and only if 
	\begin{enumerate}
	\item $\partial^1\cc{\partial}^1[J^{-1}]_1=0$
	\item $\cc{\partial}^1[J^{-1}]_1-\partial^1[J^{-1}]_1$ is $\D^1$-exact in $E_{3}^{(\bullet,\bullet),\bullet}$,
	\end{enumerate}
where we interpret $J^{-1}$ as an element in $E^{(0,0),2}_0$. 
Moreover, the generalized contact structure inducing the data is of the form
	\begin{align*}
	\mathcal{L}=(K\oplus\mathrm{Ann}(K))^{\I\omega+B}
	\end{align*}
for any choice of $\omega\in J^{-1}$ real and any real $B\in\Omega^2_L(M)$ such that $[B]_0$ is closed and 
$\D^1[[B]_0]_1=\I(\cc{\partial}^1[J^{-1}]_1-\partial^1[J^{-1}]_1)$.
\end{theorem} 

\begin{corollary}\label{Cor: ObsCls}
Let $L\to M$ be a line bundle and let $(J,K)$ be a transversally complex Jacobi structure. If 
$[J^{-1}]_1=0$, then the data comes from a generalized contact structure of the form 
	\begin{align*}
	\mathcal{L}=(K\oplus\mathrm{Ann}(K))^{\I\omega},
	\end{align*}
where $\omega\in J^{-1}$ is a real extension of the inverse of $J^{-1}$.
\end{corollary}

\begin{remark}[Generalized Complex Geometry]
Let us recall the mirror result in generalized complex geometry. In \cite{SymplFol} the author obtains similar results, 
given a regular Poisson structure $\pi\in\Secinfty(\Anti^2 TM)$ and a transversally complex involutive distribution 
$K\subseteq T_\mathbb{C} M$. Using the same formalism introduced in Subsection \ref{SubSec: SpecSeq}, we find that the data 
comes from a generalized complex structure, if and only if 
	\begin{enumerate}
	\item $\partial^1\cc{\partial}^1[\pi^{-1}]_1=0$
	\item $\cc{\partial}^1[\pi^{-1}]_1-\partial^1[\pi^{-1}]_1$ is $\D^1$-exact in $E_{3}^{(\bullet,\bullet),\bullet}$.
	\end{enumerate}	   
These obstructions differ quite a lot from those found in \cite{SymplFol}. The reason for this is that in \cite{SymplFol} the 
author searches for $H$-generalized complex structures, while we are mainly interested in  the obstructions for the existence 
of honest generalized complex 
structures. In view of Remark  \ref{Rem: Equivalence H-genCon} this is not a difference in the case of generalized contact 
geometry, in fact it is  in generalized complex geometry: not every $H$-generalized complex structure is a transformation 
of a generalized complex structure via a 2-form in the sense of Lemma \ref{Lem: TransGenCon}. It is an easy exercise to see 
that there is an $H$-generalized complex structure inducing $(\pi,K)$, if and only if 
	\begin{enumerate}
	\item $\D^1(\cc{\partial}^1[\pi^{-1}]_1-\partial^1[\pi^{-1}]_1)=0$
	\item $\D^2 [\cc{\partial}^1[\pi^{-1}]_1-\partial^1[\pi^{-1}]_1]_2=0$
	\item $\D^3 [[\cc{\partial}^1[\pi^{-1}]_1-\partial^1[\pi^{-1}]_1]_2]_3=0$.
	\end{enumerate}
To be more precise \textit{ii.)} is only well-defined, if \textit{i.)} is fulfilled and \textit{iii.)} is only well-defined, if  \textit{ii.)} is fulfilled. 
These
obstructions are equivalent to the ones found in \cite{SymplFol}. It is a bit of a computational effort to prove this, since 
the author used a transversal to $\image(\pi^\sharp)$ to obtain his results. We want to stress that, as we work completely within the spectral sequence, this is not really 
necessary and we prefer not to make this arbitrary choice.
\end{remark}

\section{Examples I: Five dimensional nilmanifolds}
Jacobi structures appear, among many 
other situations, as invariant Jacobi structures on Lie groups, which are canonically regular, and hence weakly regular. 
We begin this section describing  invariant Jacobi structures. Afterwards, we will formulate everything at the level of Lie 
algebras.

\begin{definition}
Let $L\to G$ be a line bundle over a Lie group $G$ and let $\Phi\colon G\to \Aut(L)$ be a group action covering the left 
multiplication. A Jacobi bracket $\{-,-\}\colon \Secinfty(L)\times\Secinfty(L)\to\Secinfty(L)$ is said to be invariant, if
	\begin{align*}
	\Phi^*_g\{\lambda,\mu\}=\{\Phi^*_g\lambda,\Phi^*_g\mu\} \ \forall\lambda,\mu\in\Secinfty(L), \ \forall g\in G.
	\end{align*}
\end{definition}

Our arena is the omni-Lie algebroid $DL\oplus J^1L$ for a line bundle 
$L\to G$ over a Lie group. For a Lie group action $\Phi\colon G\to \Aut(L)$, we have the the canonical action 
	\begin{align*}
	\mathbb{D}\Phi_g\colon \mathbb{D}L\ni (\Delta,\psi)\mapsto (D\Phi_g(\Delta), (D\Phi_{g^{-1}})^*\psi)\in \mathbb{D}L.   
	\end{align*}	 
\begin{definition}
Let $L\to G$ be a line bundle over a Lie group $G$, let $H\in\Secinfty(\Anti^3(DL)\otimes L)$ be closed and let $\Phi\colon 
G\to \Aut(L)$ be a Lie group action covering the left multiplication. A $H$-generalized contact
 structure $\mathcal{L}\subseteq \mathbb{D}_\mathbb{C}L$ is 
said to be $G$-invariant, if and only if $\mathbb{D}\Phi_g(\mathcal{L})=\mathcal{L}$ for all $g\in G$ and $H=\Phi_g^*H$. 
\end{definition}

\begin{proposition}
Let $L\to G$ be a line bundle over a Lie group $G$, let $H\in\Secinfty(\Anti^3(DL)\otimes L)$ be closed, let $\Phi\colon 
G\to \Aut(L)$ be a Lie group action covering the left multiplication and let  $\mathcal{L}\subseteq \mathbb{D}_\mathbb{C}L$ be a $G$-invariant $H$-generalized contact structure, then its Jacobi-structure is $G$-invariant. 
 \end{proposition}

Let us now trivialize, with the help of the Lie group action $\Phi\colon G\to \Aut(L)$, the line bundle itself, its derivations and its first jet. Similarly to the tangent bundle of a Lie group, we have 
	\begin{align*}
	L\cong G\times \ell
	\end{align*}
where $\ell=\Secinfty(L)^G$. Note that $\ell$ is a 1-dimensional vector space over $\mathbb{R}$. Moreover, the action of $G$ looks in this trivialization like
	\begin{align*}
	\Phi_g(h,l)=(gh,l) \ \forall (h,l)\in L , \ \forall g\in G.
	\end{align*}
Additionally the Atiyah algeroid is also a trivial vector bundle by 
	\begin{align*}
	DL\cong G\times \Secinfty(DL)^G, 
	\end{align*}
where $\Secinfty(DL)^G$ is a $(\dim(G)+1)$-dimensional vector space over $\mathbb{R}$. Moreover, since the symbol maps invariant derivations to left-invariant vector fields, we have the \emph{$G$-invariant 
Spencer sequence} 
	\begin{align*}
	0\to \mathbb{R}\to\Secinfty(DL)^G\to \lie{g}\to 0,
	\end{align*}
by using the fact that $G$-invariant Endomorphisms are just
multiplications by constants. This sequence splits canonically, since $\Secinfty(DL)^G$ acts as derivations on $\ell$, we have $\Secinfty(L)\cong 
\Cinfty(M)\otimes_\mathbb{R}\ell$. Thus we have 
	\begin{align*}
	\Secinfty(DL)^G\cong\lie g \oplus \mathbb{R}, 
	\end{align*}
with bracket 
	\begin{align*}
	[(\xi,r),(\eta,k)]=([\xi,\eta],0)\ \forall(\xi,r),(\eta,k)\in\lie g\oplus\mathbb{R}.
	\end{align*}
Similarly, using $J^1L=(DL)^*\otimes L$,  plus the choice of a basis of $\ell$, we get 
	\begin{align*}
	\Secinfty(J^1L)^G \cong \lie g^*\oplus \mathbb{R}.
	\end{align*}
 The differential $\D_L$ reduces two 
	\begin{align*}
	\D_L(\alpha +k\mathbbm{1}^*)=\delta_{CE}\alpha +\mathbbm{1}^*\wedge \alpha,
	\end{align*}		
where $\alpha\in \lie g^*$ and $\mathbbm{1}^*$ is the projection to the $\mathbb{R}$-component. 
These are all the ingredients, we need to describe $G$-invariant generalized contact structures via their infinitesimal data,
i.e. in terms of the Lie algebra $\lie{g}=\mathrm{Lie}(G)$.

\begin{proposition}
Let $L\to G$ be a line bundle over a Lie group $G$, let $H\in\Secinfty(\Anti^3(DL)\otimes L)$ be closed and let $\Phi\colon 
G\to \Aut(L)$ be a Lie group action. A $G$-invariant $H$-generalized contact structure $\mathcal{L}\subseteq \mathbb{D}_
\mathbb{C}L$ is 
equivalently described by a subspace $\mathcal{L}^\lie g\subseteq [(\lie g\oplus \mathbb{R})\oplus (\lie g^*\oplus
\mathbb{R})]_\mathbb{C}$, which is $H$-involutive, maximally isotropic  and fullfills 
$\mathcal{L}^\lie g\cap \cc{\mathcal{L}^\lie g}=\{0\}$,
where all the operations on $[(\lie g\oplus \mathbb{R})\oplus (\lie g^*\oplus \mathbb{R})]_\mathbb{C}$ are given by the 
identification $\Secinfty(DL\oplus J^1 L)^G=\Secinfty(DL)^G\oplus\Secinfty(J^1L)^G\cong 
[(\lie g\oplus \mathbb{R})\oplus (\lie g^*\oplus \mathbb{R})]$. 
\end{proposition}

\begin{proof}
The proof is based on the fact that an invariant generalized contact structure is completely characterized by its invariant 
sections. 
\end{proof}

The idea is now to forget about the Lie group and perform all the existence proofs directly on the Lie algebra, having in 
mind, of course, that we can reconstruct a generalized contact structure on the Lie group by translating. Being a bit more 
precise, we give the following 

\begin{definition}
Let $\lie g$ be a Lie algebra with the abelian extension $\lie g_{\mathbb{R}}:=\lie g\oplus \mathbb{R}$, where we denote by 
$\mathbbm{1}$
and $\mathbbm{1}^*$ the canonical elements in $\lie g$ and $\lie g^*$, respectively, and let 
$H\in \Anti^3 (\lie{g}^*_\mathbb{R})$ be $\D$-closed. The omni-Lie algebra of $\lie g$ is the
vector space $\lie g_\mathbb{R}\oplus \lie g_\mathbb{R}^*$ together with 
\begin{enumerate}
	\item the (Dorfman-like, H-twisted) bracket
		\begin{align*}
		[\![(X_1,\psi_1) ,(X_2,\psi_2 )]\!]_H
		=([X_1,X_2],\Lie_{X_1} \psi_2- \iota_{X_2}\D\psi_1+\iota_{X_1}\iota_{X_2}H)
		\end{align*}		 
	\item the non-degenerate pairing 
		\begin{align*}
		\bla (X_1,\psi_1) ,(X_2,\psi_2 )\bra := \psi_1(X_2)+\psi_2(X_1)
		\end{align*}
	\item the canonical projection $\pr_D\colon \lie g_\mathbb{R}\oplus \lie g_\mathbb{R}^* \to \lie g_\mathbb{R}$
	\end{enumerate}	
Here the differential is given by $\D=\delta_{CE}+\mathbbm{1}^*\wedge$ and $\Lie_X=[\iota_X,\D]$.  
\end{definition}

The obvious way to define now a generalized contact Lie algebra is the following

\begin{definition} \label{Def: GenConLieAlg}
Let $\lie g$ be a Lie algebra. A generalized contact structure on $\lie g$ is a subbundle $\mathcal{L}\in( \lie g_\mathbb{R}
\oplus \lie g_\mathbb{R}^*)_\mathbb{C}$, which is involutive, maximally isotropic and fulfils $\mathcal{L}\cap
\cc{\mathcal{L}}=\{0\}$. 
\end{definition}

From the above discussion, we can immediatly obtain
\begin{lemma}
Let $G$ be a Lie group with lie algebra $\lie{g}$. Then there is a $1:1$-correspondence of  generalized contact structures on $\lie g$ and left-invariant generalized contact structures on $G\times \mathbb{R}\to G$ given by the left-translation.
\end{lemma}

As in the geometric setting we have extreme cases 

\begin{example}\label{Ex: ConLieAlg}
Let $(\lie g,\Theta)$ be a $(2n+1)$ contact Lie algebra, i.e.  $\Theta\in\lie g^*$, such that 
$\Theta\wedge(\delta_{CE}\Theta)^n\neq 0$, then we denote by $\Omega=p^*\Theta$ for $p\colon\lie g_{\mathbb{R}}\to \lie g$ 
and get that
	\begin{align*}
	\mathcal{L}=\{(X,\I\iota_{X}\D\Omega)\in  (\lie g_\mathbb{R}\oplus \lie g_\mathbb{R}^*)_\mathbb{C}
	\ | \ X\in(\lie g_\mathbb{R})_\mathbb{C}\}	
	\end{align*}	  
gives $\lie g$ the structure of a generalized contact Lie algebra. 
\end{example}

\begin{example}\label{Ex: ComLieAlg}
Let $\lie g$ be a Lie algebra and $\phi\in\End(\lie g_\mathbb{R})$ be a complex structure, then 
\begin{align*}
\mathcal{L}=\lie g_\mathbb{R}^{(1,0)}\oplus \mathrm{Ann}(\lie g_\mathbb{R}^{(1,0)})
\end{align*}
gives $\lie g $ the structure of a generalized contact Lie algebra, where $\lie{g}_\mathbb{R}^{(1,0)}$ is the 
$+\I$-Eigenbundle of $\phi$. 
\end{example}

We shrink ourselves now to the case of 5-dimensional nilpotent Lie algebras, since we are able to use already existing classification results, which are not available in more general classes of Lie algebras.   
To be more precise, we want to prove the following 
\begin{theorem}\label{Thm: NilLieALgGenCOn}
Every five dimensional nilpotent Lie algebra possesses a generalized contact structure.
\end{theorem} 
From this theorem, we can immediately conclude
\begin{corollary}
Every five dimensional nilmanifold possesses an invariant generalized contact structure.   
\end{corollary}

To prove Theorem \ref{Thm: NilLieALgGenCOn}, we will use Section \ref{Sec: ObsTransComJac}. In particular, we will find  a 
generalized contact structure on a given Lie algebra by looking for  a transversally complex Jacobi 
structure. Afterwards, we use Theorem \ref{Thm: ObsCls} to prove the existence of a generalized contact structure. Note that 
we did not prove the invariant analogue of Theorem \ref{Thm: ObsCls}, but as the proof suggests this can be done with a bit 
of effort. 

A big help in proving Theorem \ref{Thm: NilLieALgGenCOn} is the classification of five dimensional 
nilpotent Lie algebras provided in 
\cite{DEGRAAF}. In that work the author proved that there are exactly nine (isomorphism classes of) five dimensional 
nilpotent 
Lie algebras. 
Since we want to prove that there are generalized contact structures on all of them, it seems 
convenient to test first the extreme examples, i.e. integrable complex structures on $\lie g_\mathbb{R}$ 
(Example \ref{Ex: ComLieAlg}) on the one hand 
and contact structures on the other hand (Example \ref{Ex: ConLieAlg}). For the complex structures we can can use the work of 
Salamon in \cite{SALAMON}, 
where he classified all the complex nilpotent Lie algebras of six dimensions. In the following, we denote  by $\{e_1,
\dots,e_5\}$ a given basis of a five dimensional vector space $\lie g$. The only 5-dimensional nilpotent Lie algebras,
such that $\lie g_\mathbb{R}$ admits a complex structure are (we use the notation of \cite{DEGRAAF} for his description of 5-
dimesional nilpotent Lie algebras):  

\begin{enumerate}
\item $\mathfrak{L}_{5,1}$ (abelian)
\item $\mathfrak{L}_{5,2}: \ [e_1,e_2]=e_3$
\item $\mathfrak{L}_{5,4}: \ [e_1,e_2]=e_5,\ [e_3,e_4]=e_5 $
\item $\mathfrak{L}_{5,5}: \ [e_1,e_2]=e_3,\ [e_1,e_3]=e_5,\ [e_2,e_4]=e_5 $
\item $\mathfrak{L}_{5,8}: \ [e_1,e_2]=e_4,\ [e_1,e_3]=e_5$
\item $\mathfrak{L}_{5,9}: \ [e_1,e_2]=e_3,\ [e_1,e_3]=e_4, \ [e_2,e_3]=e_5$
\end{enumerate}

We have to check that the remaining 5-dimensional nilpotent Lie algebras $\mathfrak{L}_{5,3}$, $\mathfrak{L}_{5,6}$ and $
\mathfrak{L}_{5,7}$ are also generalized 
contact. Let us denote by $\{e^1,\dots,e^5\}$ the dual basis of  $\{e_1,\dots,e_5\}$.

\begin{subsection}{$\mathfrak{L}_{5,3}: [e_1,e_2]=e_3,\ [e_1,e_3]=e_4$}
It is easy to see that $J=e_3\wedge e_1+ \mathbbm{1}\wedge e_4$ is a Jacobi structure. Additionally
	\begin{align*}
	K:=\langle \mathbbm{1},e_1,e_3,e_4,e_2-\I e_5 \rangle
	\end{align*}
is a subalgebra of $(\lie{g}_\mathbb{R})_\mathbb{C}$ and  that $(J,K)$ is a transversally complex Jacobi structure.
Moreover, $\omega=e^1\wedge e^3 -\mathbbm{1}^*\wedge e^4$ is an extension of the inverse of $J$ and we obtain that $\D\omega=
\delta_{CE}
\omega+\mathbbm{1}^*\wedge \omega=0$, which implies that $[J^{-1}]_1=[[\omega]_0]_1=0$.  Using Corollary 
\ref{Cor: ObsCls}, we see that there is a generalized contact structure inducing this data, an explicit example is 
given by 
	\begin{align*}
	(K\oplus\mathrm{Ann}(K))^{\I\omega}.
	\end{align*}
\end{subsection}

\subsection{$\mathfrak{L}_{5,6}\colon [e_1,e_2]=e_3,\ [e_1,e_3]=e_4,\ [e_1,e_4]=e_5, \ [e_2,e_3]=e_5$}
This Lie algebra is actually a contact Lie algebra with contact 1-form $\Theta=e^5$.

\begin{subsection}{$\mathfrak{L}_{5,7}\colon [e_1,e_2]=e_3,\ [e_1,e_3]=e_4,\ [e_1,e_4]=e_5$}\label{SubSec: Lie7}
It is easy to see that  $J=e_1\wedge e_3 +e_4\wedge(\mathbbm{1}+e_5)$ is a Jacobi structure. Let us define 
\begin{align*}
K:=\langle \mathbbm{1}+e_5, e_1,e_3,e_4, \mathbbm{1}+\I e_2\rangle.
\end{align*}
We have $[\lie g_\mathbb{R}, \lie g_\mathbb{R} ]\subseteq \image(J^\sharp)$ and hence integrability of the 
corresponding $K$ is canonically fulfilled. Moreover, we have that $\omega=-(e^1\wedge e^3+ e^4\wedge e^5)+(\mathbbm{1}-e^5)
\wedge e^4$ is an extension of $J^{-1}$ 
and is closed with respect to $\D=\delta_{CE}+\mathbbm{1}\wedge$. Using Corollary \ref{Cor: ObsCls}, we find a 
generalized contact structure given by 
	\begin{align*}
	(K\oplus\mathrm{Ann}(K))^{\I\omega}.
	\end{align*}
	
\end{subsection}  

\

We have already seen that the Lie algebras  $\mathfrak{L}_{5,3}$, $\mathfrak{L}_{5,6}$ and $\mathfrak{L}_{5,7}$ do not admit 
a complex structure on their one dimensional abelian extension. Moreover,  $\mathfrak{L}_{5,6}$ is a contact Lie algebra. 
In the following we want to show that $\mathfrak{L}_{5,3}$ and $\mathfrak{L}_{5,7}$ do not admit a contact structure, so 
that there are generalized contact structures on them but not of the extreme types. 
Let us first collect some basic properties of contact Lie algebras

\begin{theorem}\label{Thm: KerCon}
Let $\lie g$ be a nilpotent Lie algebra and $\Theta\in\lie g^*$ be a contact form. Then the center $Z(\lie g)$ has dimension 
one.
\end{theorem}
A reference of Theorem \ref{Thm: KerCon} and of its proof is \cite{Remm}. 
As a first consequence we have 
\begin{corollary}
The Lie algebra $\mathfrak{L}_{5,3}$ is not contact. 
\end{corollary}

The only Lie algebra, which is left over is $\mathfrak{L}_{5,7}$. Here we do not have a general statement about contact 
Lie algebras that we can use, nevertheless we have
\begin{lemma}
The Lie algebra $\mathfrak{L}_{5,7}$ is not contact.
\end{lemma}

\begin{proof}
From the commutation relation in Subsection \ref{SubSec: Lie7} it is clear that we have $\delta_{CE}(\Anti^\bullet\lie g^*)
\subseteq e^1\wedge\Anti^\bullet \lie g^*$. Hence we have that for all $\alpha\in \lie g^*$ $\delta_{CE}
\alpha=e^1\wedge\beta$ for some $\beta\in\lie g^*$. As a consequence $\alpha\wedge(\delta_{CE}\alpha)^2=0$ for all $\alpha
\in \lie g^*$, and hence the Lie algebra can not be contact.  
\end{proof}

\begin{remark}
To prove that $\mathfrak{L}_{5,1}$, $\mathfrak{L}_{5,2}$, $\mathfrak{L}_{5,8}$ and $\mathfrak{L}_{5,9}$ are not contact it 
is enough to
obtain that all of them can be excluded by Theorem \ref{Thm: KerCon}. Moreover, for the remaining ones the contact 
structures are given by $e^5$.  
\end{remark}

As a summary we have the following table

\begin{center}
\begin{tabular}
[c]{c|c|c|c}
 &  contact  & $\lie g_\mathbb{R}$-complex & generalized contact \\ \hline\hline
 \vspace{-9pt}&  &  & \\
 $\mathfrak L_{5,1}$   & $\times$ & $\checkmark$ & $\checkmark$ \\ \hline
 \vspace{-9pt}&  & &  \\
$\mathfrak L_{5,2}$ & $\times$ & $\checkmark$ & $\checkmark$ \\ \hline
 \vspace{-9pt} &  & &  \\
$\mathfrak L_{5,3}$  & $\times$ & $\times$ &$\checkmark$ \\ \hline
 \vspace{-9pt} &   & & \\
 $\mathfrak L_{5,4}$  & $\checkmark$ &$\checkmark$ &$\checkmark$ \\ \hline
  \vspace{-9pt}&   & & \\
$\mathfrak L_{5,5}$  & $\checkmark$  &$\checkmark$ &$\checkmark$ \\ \hline
 \vspace{-9pt}&  & & \\
$\mathfrak L_{5,6}$  & $\checkmark$  & $\times$ &$\checkmark$ \\ \hline
 \vspace{-9pt}&  &  &\\
$\mathfrak L_{5,7}$  & $\times$  & $\times$ & $\checkmark$ \\ \hline
\vspace{-9pt}&  & & \\
$\mathfrak L_{5,8}$    & $\times$ &$\checkmark$ &$\checkmark$ \\ \hline
 \vspace{-9pt}&  &  &\\
$\mathfrak L_{5,9}$  & $\times$ & $\checkmark$&$\checkmark$ \\ \hline
\end{tabular}

\end{center}

\section{Examples II: Contact Fiber Bundles}

The next class of examples are \emph{contact fiber bundles} over a complex base manifold. We begin explaining what we 
mean by 
contact fiber bundle.  Similarly to symplectic fiber bundles, there is a contact structure on the vertical bundle 
	\begin{align*}
	\mathrm{Ver}_L(P)=\sigma^{-1}(\mathrm{Ver}(P))\subseteq DL,
	\end{align*}
for a line bundle $L\to P$, such that $P\to M$ is a fiber bundle. More precisely:
\begin{definition}
Let $\pi\colon P\to M$ be a fiber bundle  and let $L\to P$ by a line bundle. A
smooth family of contact manifolds is the data of $L\to P$ together with a closed non-degenerate 2-form $\omega\in
\Secinfty(\Anti^2(\mathrm{Ver}_L(P))^*\otimes L)$. 
If additionally the contact structures $(L\at{P_m}\to 
P_m,\omega\at{D(L|_{P_m})})$ are contactomorphic, we say that $L\to P$ is a contact fiber 
bundle.    
\end{definition}

Before we come to examples, we want to make some general remarks on smooth families of contact structures and contact fiber 
bundles, which are more or less known. 

\begin{remark}
Let $(L\to P, \pi\colon P\to M, \omega)$ be a smooth family of contact structures. If the fiber is compact and connected and the base is connected, 
then the data automatically is a contact fiber bundle. This follows from the stability theorem of Gray in \cite{GrayStab}, 
which states that two contact forms which are connected by a smooth path of contact structures are contactomorphic.  
\end{remark}

\begin{remark}\label{Rem: LocConBun}
As in the setting of symplectic fiber bundles, we can express the data in local terms, namely:
the datum of a contact fiber bundle over a manifold $M$ with typical fiber $F$  is equivalent to:
	\begin{enumerate}
	\item a line bundle $L_F\to F$ and a contact 2-form $\omega\in\Secinfty(\Anti^2 (DL)^*\otimes L)$
	\item an open cover $\{U_i\}_{i\in I}$
	\item smooth transition maps $T_{ij}\colon U_i\cap U_j\to \Aut(L)$ which are point-wise contactomorphisms
	\end{enumerate} 
\end{remark}

\begin{remark}
Obviously, one can define smooth families of contact structures as a Jacobi structure of contact type,  such that the 
characteristic 
distribution of it is the vertical bundle of a fiber bundle.
\end{remark}

Using this remarks, we can show that under certain assumptions on the base, a smooth family of contact structures always 
induces a generalized contact structure on the total space. 

\begin{lemma}\label{Lem: ConFibBun}
Let $\pi\colon P\to M$ be a fiber bundle with typical fiber $F$ over a complex base $M$, let $L\to M$ by a line bundle and 
let $J\in \Secinfty(\Anti^2(J^1 L)^*\otimes L)$ be a Jacobi 
structure giving $P$ the structure of a smooth family of contact manifolds. 
Then $P$ possesses a generalized contact structure
with Jacobi structure $J$. 
\end{lemma}

\begin{proof}
First of all, we notice that that the Jacobi structure is weakly regular, since $\image(J^\sharp)=\mathrm{Ver}_L(P)$.
The only thing what we have to show is that the data induce a transversally complex Jacobi structure, since we have that the 
inverse of the Jacobi structure is leaf-wise exact with canonical primitive $\iota_\mathbbm{1} \omega$, 
which implies that $[J^{-1}]_1=0$ by Corollary \ref{Cor: ObsCls}. Let us denote by 
$T^{(1,0)}M\subseteq T_\mathbb{C}M$ the holomorphic tangent bundle induced by the complex structure on $M$. With this we define 
	\begin{align*}
	K:= (\sigma\circ T\pi)^{-1}(T^{(1,0)}M)\subseteq D_\mathbb{C}L.
	\end{align*}	
It is an easy consequence of the definitions of the bundles that $(J,K)$ is a transversally complex Jacobi structure, and hence there exists a generalized contact structure inducing it.
\end{proof}

We see that in this case the existence of a generalized contact structure is unobstructed. Thus we want to show that 
this kind of structures exist and give some classes of examples. 

\begin{example}[Projectivized Vertical Bundle]

Let $\pi\colon P\to M$ fiber bundle  with typical fiber $F$ over a complex base $M$. Given a set of local 
trivializations 
$(U_i,\tau_i)_{i\in I}$ 
	\begin{center}
		\begin{tikzcd}
 		 P\at{U_i}  \arrow[rr, "\tau_i"]\arrow[dr, "\pi"'] & &  U_i\times F \arrow[dl, "\pr_1"] \\
 		& U_i &
		\end{tikzcd}
	\end{center}
with transition functions $\tau_{ij}\colon U_i\cap U_j\to \Diffeo(F)$ . Let us denote by 
$T_* \tau_{ij}\colon U_i\cap U_j\to \Diffeo(T^*F)$ there cotangent lifts, which fulfil also the cocycle condition for 
transition functions and hence belong to a fiber bundle $ \tilde{V} \to M$ (actually this is $\mathrm{Ver}^*(P)$) with local 
trivializations $(U_i,\phi_i)_{i\in I}$, such 
that the transition functions fulfil $\phi_{ij}= T_* \tau_{ij}$.   
We consider now the canonical symplectic form $\omega_{can}\in \Secinfty(\Anti^2 T^*(T^*F))$  on $T^*F$, note that the 
functions $T_* \tau_{ij}$ are symplectomorphisms, since they are point transformations. The next step is to consider  the 
the fiber bundle $V\to M$ with typical fiber   $T^*F\smallsetminus{0_F}$, which we get by the obvious restrictions. 
Note that on $T^*F\smallsetminus{0_F}$ we have a canonical $\mathbb{R}^\times$-action which is free and proper and the for 
restricted symplectic form  $\omega_{can}$ we have that $\Lie_{(1)_{T^*F}}\omega_{can}=\omega_{can}$. Using the results from  
\cite{2017SIGMA..13..059B}, we conclude that the associated line bundle $L \to \mathbb{R}T^*F
:=\frac {T^*F\smallsetminus{0_F}}
{\mathbb{R}^\times}$ carries a contact structure and the transitions functions, which are obviously commuting with the $
\mathbb{R}^\times$-action,  act as line bundle automorphisms preserving the contact structure. 
This is exactly the data we need to cook up a contact fiber bundle (\ref{Rem: LocConBun}) and hence its total space possess 
a generalized contact structure, due to Lemma \ref{Lem: ConFibBun}. Note that here the input was a 
generic fiber bundle over a complex base and the output is a generalized contact bundle. Moreover, if both the base and the 
fiber are compact, then the output is also compact. We hence proved the existence of compact examples. 
\end{example}

After this rather general construction of examples of contact fiber bundles, 
we want to give a more down-to-earth class which also naturally appears.   

\begin{example}[Principal fiber Bundles]
Let $\lie g$ be a Lie algebra with a contact 1-form $\Theta\in\lie g^*$. Let us consider a Lie group $G$ integrating $\lie g$ and a 
manifold $M$ with a complex structure. Additionally, let $P\to M$ be a $G$-principal fiber bundle and let
  $\mathbb{R}_P\to P$ be the trivial line bundle, 
  where we denote by $1_P$ the generating section. Note that here the gauge algebroid splits 
canonically as $D\mathbb{R}_P=TP\oplus\mathbb{R}_P$. Moreover, we have that 
$\mathrm{Ver}_{\mathbb{R}_P}(P)=\mathrm{Ver}(P)\oplus \mathbb{R}_P$. Thus, a generic derivation 
$\Delta_p\in \mathrm{Ver}_{\mathbb{R}_P}(P) $  is of the form $\Delta=(\xi_P(p), k)\in \mathrm{Ver}(P)\oplus \mathbb{R}_P $
for the fundamental vector field $\xi_P$ of a unique $\xi\in \lie g$ and for $p\in P$. We define 
$\omega \in \Secinfty(\Anti^2(\mathrm{Ver}_{\mathbb{R}_P}(P))^*\otimes {\mathbb{R}_P})$ by 
	\begin{align*}
	\omega((\xi_P(p),k), (\eta_P(p),r))= \big((\delta_{CE}\Theta)(\xi,\eta)+k\Theta(\eta)-r\Theta(\xi)\big)\cdot 1_P(p).
	\end{align*}
It is easy to check that $\omega$ gives $P\to M$ the structure of a contact fiber bundle. Since $M$ was assumed to be 
complex, we can apply  Lemma \ref{Lem: ConFibBun} to obtain a generalized contact bundle on $P$. Note that this notion 
includes $\mathbb{S}^1$-principal fiber bundles over a complex manifold. Moreover, contact Lie algebras are an active 
field of research and there are many examples around and even a classification of nilpotent contact Lie algebras in  
\cite{Alvarez}.   
\end{example}

\begin{remark}
There is also the notion of smooth families of locally conformal symplectic structures, which corresponds to a weakly regular 
Jacobi structures with just even dimensional leaves such that the projection to the leaf space is the projection map of fiber 
bundle. A transversally complex Jacobi 
structure can be induced, at least in some cases,  by a Atiyah-complex structure on the base.
In this specific case there are examples which do not 
come from generalized contact structure (one explicit counterexample is provided in Section \ref{Sec: CountEx}). 
Since the notion of Atiyah-complex structures is by far not as developed as the notion of complex structures, it is not very 
easy at the moment to construct bigger classes of counter examples.  
\end{remark}

\section{A Counterexample}\label{Sec: CountEx}
In this last section, we want to construct a transversally complex Jacobi structure which cannot be induced by a generalized 
contact structure. The remarkable feature of this counterexample is, that it is, as manifolds,  a global product of a 
(locally conformal) symplectic manifold and an Atiyah-complex manifold. Note that in \cite{2017arXiv171108310S} it was proven 
that every generalized contact bundle is locally isomorphic to a product,
however not all products arise in this way.

Let us consider the 2-sphere $\mathbb{S}^2$ and its symplectic form $\omega\in\Secinfty(\Anti^2 T^*\mathbb{S}^2)$. Its inverse 
$\pi\in \Secinfty(\Anti^2 TM)$ is a Poisson structure and hence 
$\pi+\mathbbm{1}\wedge 0=\pi$ is a Jacobi structure on the trivial line bundle.

The second manifold which is involved is the circle $\mathbb{S}^1$. Our counterexample will live on the trivial line bundle over the product 
	\begin{align*}
	\mathbb{R}_M\to M:=\mathbb{S}^2\times \mathbb{S}^1.
	\end{align*}	 
Using Remark \ref{Rem: TrivLine}, we see that 
	\begin{align*}
	D\mathbb{R}_M=TM\oplus \mathbb{R}_M= T\mathbb{S}^2\oplus T\mathbb{S}^1\oplus \mathrm{R}_M
	\end{align*}	 
and we can define a Jacobi structure $J=\pi+\mathbbm{1}\wedge 0=\pi$ on it by "pulling back". We see that 
$\image(J^\sharp)=T\mathbb{S}^2\subseteq D\mathbb{R}_M$. 

The next step is to choose an everywhere non-vanishing vector field $e\in \Secinfty(T\mathbb{S}^1)$ and define
	\begin{align*}
	K:=T_\mathbb{C}\mathbb{S}^2\oplus\langle \mathbbm{1}-\I e\rangle\subseteq D_\mathbb{C}\mathbb{R}_M.
	\end{align*}
Note that the derivation  $\mathbbm{1}-\I e\in (T\mathbb{S}^1\oplus\mathbb{R}_{\mathbb{S}^1})_\mathbb{C}\subseteq
D_\mathbb{C}\mathbb{R}_{\mathbb{S}^1}$	is the $+\I$-Eigenbundle of an 
Atiyah-complex structure on $\mathbb{R}_{\mathbb{S}^1}\to \mathbb{S}^1$. Moreover, we have  
	\begin{align}\label{Eq: Split}
	D_\mathbb{C}\mathbb{R}_M=T_\mathbb{C}\mathbb{S}^2\oplus \langle \mathbbm{1}-\I e\rangle
	\oplus\langle \mathbbm{1}+\I e\rangle.
	\end{align} 	 
An easy computation shows that $(J,K)$ is a transversally complex Jacobi structure. Our claim is now that $(J,K)$ can not be 
induced by a generalized contact structure. To see this, let us examine the $Der$-complex a bit closer. We have that 
	\begin{align*}
	\Anti^k(D\mathbb{R}_M)^*\otimes \mathbb{R}_M=\Anti^k(TM^*\oplus\mathbb{R}_M).
	\end{align*}
Using the notation of Remark \ref{Rem: TrivLine}, we obtain that a $\psi\in \Secinfty(\Anti^k(TM^*\oplus\mathbb{R}_M))$ can 
be uniquely 
written as $\psi=\alpha+\mathbbm{1}^*\wedge \beta$ for some $(\alpha,\beta)\in \Secinfty(\Anti^k T^*M\oplus \Anti^{k-1} T^*M)
$. It is an easy verification to show that for the differential $\D_{\mathbb{R}_M}$ we have 
	\begin{align*}
	\D_{\mathbb{R}_M}(\alpha+\mathbbm{1}^*\wedge \beta)=\D\alpha+\mathbbm{1}^*\wedge(\alpha-\D\beta)
	\end{align*}
where $\D$ is the usual de Rham differential, 
 \cite[Remark 1.1.1]{2017arXiv171108310S}. Now we want to pass to the spectral sequence, therefore we split the $Der$-complex 
according to the splitting of Equation \ref{Eq: Split}, we have 
	\begin{align*}
	\Omega_{\mathbb{R}_M}^{(i,j),q}
	= \Secinfty(\Anti^q T^*\mathbb{S}^2 \otimes \Anti^i \langle \mathbbm{1}^*+\I\alpha\rangle \otimes  
	\Anti^j\langle \mathbbm{1}^*-\I\alpha\rangle),
	\end{align*}	  
where $\alpha\in \Secinfty(T\mathbb{S}^1)$ such that $\alpha(e)=1$.
Note that here we have the canonical identification 
$D\mathbb{R}_M\supseteq\mathrm{Ann}(T\mathbb{S}^2)=T^*\mathbb{S}^1\oplus \mathbb{R}_M$, which 
allows us to identify 
	\begin{align*}
	E_0^{(i,j),q}=\Omega_{\mathbb{R}_M}^{(i,j),q}.
	\end{align*}
In case of a product, the differential $\D_{\mathbb{R}_M}$ splits canonically with respect to the bi-grading into
	\begin{align*}
	\D_L=\D^0+\partial^1+\cc{\partial}^1,
	\end{align*}
where $\D^0\colon\Omega_{\mathbb{R}_M}^{(i,j),q}\to \Omega_{\mathbb{R}_M}^{(i,j),q+1}$, $\partial^1\colon
\Omega_{\mathbb{R}_M}^{(i,j),q}\to \Omega_{\mathbb{R}_M}^{(i+1,j),q}$ and $\cc{\partial}^1\colon \Omega_{\mathbb{R}_M}
^{(i,j),q}\to \Omega_{\mathbb{R}_M}^{(i,j+1),q}$. Additionally, all three maps are differentials themselves and anticommute
pairwise.
Now we want to consider the inverse of $J$, which is the pullback of $\omega$ with respect to the canonical projection $
\mathbb{S}^2\times\mathbb{S}^1\to\mathbb{S}^2$. With a tiny abuse of notation we will see 
$\omega$ is an element of $\Secinfty(\Anti^2 T^*\mathbb{S}^2)\subseteq E_0^{(0,0),2}$.
A long and not very enlightening computation shows that 
	\begin{align*}
	\partial^1\cc{\partial}^1\omega=\frac{1}{4}( (\mathbbm{1}^*-\I\alpha)\wedge( \mathbbm{1}^*+\I\alpha)\wedge\omega)
	\end{align*}
Hence, we have for the cohomology class 
	\begin{align*}
	\partial^1\cc{\partial} ^1[[\omega]_0]_1=[[ (\mathbbm{1}^*-\I\alpha)\wedge( \mathbbm{1}^*+\I\alpha)\wedge\omega]_0]_1.
	\end{align*}
But this cannot vanish, since a $\D^0$-primitive $\psi$ has to be of the form 
	\begin{align*}
	\psi =(\mathbbm{1}^*-\I\alpha)\wedge( \mathbbm{1}^*+\I\alpha)\wedge\beta
	\end{align*}
for $\beta\in \Secinfty(T^*\mathbb{S}^2)$, which implies that $\D\beta=\omega$, 
which is an absurd because the symplectic form on the sphere is not exact.

\end{document}